\documentclass[11pt]{amsart}

\usepackage[T1]{fontenc}
\usepackage{lmodern}
\usepackage[margin=1.12in]{geometry}
\usepackage{amsmath,amssymb,amsthm,mathtools}
\usepackage{microtype}
\usepackage[colorlinks=true,linkcolor=blue,citecolor=blue,urlcolor=blue]{hyperref}

\newtheorem{theorem}{Theorem}[section]
\newtheorem{proposition}[theorem]{Proposition}
\newtheorem{lemma}[theorem]{Lemma}
\newtheorem{corollary}[theorem]{Corollary}
\theoremstyle{definition}
\newtheorem{definition}[theorem]{Definition}

\theoremstyle{remark}
\newtheorem{remark}[theorem]{Remark}

\newcommand{\C}{\mathbb C}

\newcommand{\Rm}{\operatorname{Rm}}
\newcommand{\II}{\mathrm{II}}

\newcommand{\eps}{\varepsilon}

\newcommand{\Sym}{\operatorname{Sym}}

\title[Curvature growth and holomorphic splitting]
{Curvature growth and holomorphic splitting for K\"ahler metrics with cone singularities}
\author{Martin de Borbon}
\date{\today}
\address{Department of Mathematical Sciences, Loughborough University, Schofield Building, Loughborough LE11 3TU, United Kingdom}
\email{m.de-borbon@lboro.ac.uk}

\subjclass[2020]{Primary 53C55; Secondary 32Q20, 53C25}
\keywords{K\"ahler metrics with cone singularities, second fundamental form, normal bundle sequence, Bochner normal coordinates}

\begin{document}

\begin{abstract}
	We study K\"ahler metrics with cone angle \(2\pi\beta\) along a smooth divisor, in the range \(1/2<\beta<1\). We show that bounded Riemann curvature near the divisor forces the normal bundle sequence to split holomorphically, confirming a condition proposed by Donaldson.
\end{abstract}

\maketitle

\section{Introduction}

K\"ahler metrics with cone singularities along a smooth divisor arise naturally throughout complex differential geometry; see, for example, \cite{Rubinstein2014} and the references therein. A basic local regularity question is whether the Riemann curvature remains bounded near the divisor when the cone angle is \(2\pi\beta\), with
\begin{equation}\label{eq:beta-range}
	\frac{1}{2} < \beta < 1.
\end{equation}

Arezzo--Della Vedova--La Nave \cite{ADVLN2018} proved the existence of bounded-curvature cone metrics under the assumption that the divisor admits a holomorphic tubular neighbourhood. They also recorded a suggestion of Donaldson \cite[p.~282]{ADVLN2018} that the optimal holomorphic condition for the existence of such metrics should instead be the splitting of the normal bundle sequence
\begin{equation}\label{eq:normal-sequence-intro}
	0\longrightarrow TD\longrightarrow TX|_D\longrightarrow N_D\longrightarrow0.
\end{equation}
In this paper we prove the necessary direction: bounded curvature forces the sequence \eqref{eq:normal-sequence-intro} to split holomorphically.

\subsection*{Outline}

Section~\ref{sec:admissible} introduces admissible metrics (Definition~\ref{def:admissible}), proves their invariance under divisor-preserving holomorphic coordinate changes (Lemma~\ref{lem:stability-admissibility}), and associates to their leading coefficients a boundary Hermitian metric \(h\) on \(TX|_D\) (Propositions~\ref{prop:boundary-data} and~\ref{prop:boundary-hermitian-metric}). Proposition~\ref{prop:relative-Bochner} constructs normal coordinates adapted to \(D\), in which the second fundamental form \(\II_h\) of \(TD\subset TX|_D\) is precisely the first term obstructing the metric from having the normal form of the standard flat cone metric. Section~\ref{sec:curvature} computes the leading mixed curvature (Theorem~\ref{thm:main}); this gives the curvature growth when \(\II_h\neq0\) (Corollary~\ref{cor:growth}) and shows that bounded curvature forces the normal bundle sequence \eqref{eq:normal-sequence-intro} to split holomorphically (Corollary~\ref{cor:splitting}).

\subsection*{Acknowledgements}
I learned about K\"ahler metrics with cone singularities, and in particular
about the boundary Hermitian metric \(h\) on \(TX|_D\), from Simon Donaldson. I am
very grateful to him for his beautiful and inspiring teaching. I also want to thank Yang Li for a discussion on admissible metrics and adapted normal coordinates during his visit to Loughborough in 2025.

\section{Boundary Hermitian metric and adapted normal coordinates}\label{sec:admissible}

\subsection{Admissible metrics}
Fix a smooth divisor \(D\subset X\). An \emph{adapted holomorphic coordinate system} is a system
\[
 (z^1,z^2,\ldots,z^n)=(z,y^2,\ldots,y^n)
\]
in which \(D=\{z=0\}\). We use the index \(1\) for the normal direction and indices \(a,b,c,d\) for tangential directions. Set \(\rho=|z|\).
The standard flat cone metric of angle \(2\pi\beta\) is
\begin{equation}\label{eq:model-cone}
 g_\beta=\rho^{2\beta-2}|dz|^2+\sum_{a=2}^n|dy^a|^2.
\end{equation}
We use the standard terminology that a K\"ahler metric \(g\) on \(X \setminus D\) has \emph{cone angle \(2\pi\beta\) along \(D\)} when, near every point of \(D\), it is uniformly quasi-isometric in adapted holomorphic coordinates to \eqref{eq:model-cone}.

\begin{definition}[Admissible metrics]\label{def:admissible}
Fix \(\beta \in (1/2,1)\).
A K\"ahler metric \(g\) of cone angle \(2\pi\beta\) along \(D\) is called \emph{admissible} if, near every \(p\in D\), there are adapted holomorphic coordinates \((z,y)\) and a real local K\"ahler potential \(\Phi\) for \(g\) of the form
\begin{equation}\label{eq:admissible-potential}
 \Phi
 =A(y,\bar y)+\bar z\,Q(y,\bar y)+z\,\overline{Q(y,\bar y)}
 +\frac{s(y,\bar y)^\beta}{\beta^2}|z|^{2\beta}+\Psi,
\end{equation}
where \(A=A(y,\bar y)\) is smooth and real-valued, \(Q=Q(y,\bar y)\) is smooth and complex-valued, and \(s=s(y,\bar y)\) is smooth, real-valued, and strictly positive. Moreover, the Hermitian matrix
\[
 H_{a\bar b}:=\partial_a\partial_{\bar b}A
\]
is positive definite. The remainder \(\Psi\) is real-valued and smooth away from \(D\). We write \(f=O_{\mathrm{con}}(\rho^\mu)\) if, after applying any finite product of the vector fields
\[
 z\partial_z,\qquad \bar z\partial_{\bar z},\qquad
 \partial_{y^a},\qquad \partial_{\bar y^a},
\]
the resulting function is \(O(\rho^\mu)\). We require that there is \(\eps>0\) such that
\begin{equation}\label{eq:admissible-remainder}
 \Psi=O_{\mathrm{con}}\bigl(\rho^{2\beta+\beta\eps}\bigr).
\end{equation}
\end{definition}

\begin{remark}\label{rem:epsilon-choice}
After decreasing \(\eps>0\) if necessary, we may always assume
\begin{equation}\label{eq:epsilon-range}
 2\beta+\beta\eps<2.
\end{equation}
With this choice, ordinary smooth terms of order \(O(|z|^2)\) are allowed in \(\Psi\).
\end{remark}

\begin{remark}\label{rem:normal-derivatives}
Since \(\partial_z=z^{-1}(z\partial_z)\), and similarly for \(\partial_{\bar z}\), \eqref{eq:admissible-remainder} implies the corresponding ordinary normal derivative estimates of every order, with one power of \(\rho^{-1}\) lost for each derivative. The curvature calculation in Section~\ref{sec:curvature} uses only derivatives of total normal order at most four.
\end{remark}

\begin{remark}
Since \(2\beta>1\), we have \(\Psi=o(|z|)\), so the terms linear in \(z,\bar z\) cannot be absorbed into \(\Psi\). Changing the K\"ahler potential \(\Phi\) by adding a pluriharmonic term \(zf(y)+\overline{z f(y)}\), with \(f\) holomorphic, changes \(Q\) by addition of the antiholomorphic function \(\overline{f(y)}\); see the proof of Proposition~\ref{prop:boundary-data}.
\end{remark}

\begin{remark}\label{rem:JMR}
K\"ahler--Einstein metrics with cone singularities of angle \(2\pi\beta\in(\pi,2\pi)\) are admissible in the sense of Definition~\ref{def:admissible}. This follows from the polyhomogeneous regularity theory developed in \cite{JMR2016}; see in particular \cite[Theorem~2 and Proposition~4.4, equation~(57)]{JMR2016}.
\end{remark}

\subsection{The boundary Hermitian metric}

Let \(g\) be admissible. In adapted holomorphic coordinates \((z,y)\), let
\(\Phi\) be a local K\"ahler potential as in
\eqref{eq:admissible-potential}, and use the notation \(A,Q,s,\Psi\) and
\(H_{a\bar b}\) from Definition~\ref{def:admissible}. Set
\begin{equation}\label{eq:q-def}
 q_a:=\partial_aQ.
\end{equation}

\begin{proposition}\label{prop:boundary-data}
The local data \(H=(H_{a\bar b})\), \(s\), and the complex line spanned by
\begin{equation}\label{eq:Lg-local}
 \ell=\partial_z-\overline{q_b}H^{a\bar b}\partial_{y^a}
\end{equation}
define intrinsically along \(D\) a K\"ahler metric \(h_D\) on \(TD\), a
Hermitian metric \(h_N\) on \(N_D=TX|_D/TD\) satisfying
\[
 |[\partial_z]|_{h_N}^2=s,
\]
and a smooth complex line subbundle \(L_g\subset TX|_D\) transverse to
\(TD\).
\end{proposition}

\begin{proof}
Changing the K\"ahler potential by the real part of a holomorphic
function changes \(A\) by a tangential pluriharmonic function and \(Q\)
by an antiholomorphic function of \(y\). Hence \(H\), \(q\), and \(s\)
are unchanged.

Let \((w,x)\) be another system of adapted holomorphic coordinates. Along
\(D=\{w=0\}\), write the inverse coordinate change to first order as
\[
 z=u(x)w+O(w^2),
 \qquad
 y^a=F^a(x)+wW^a(x)+O(w^2),
\]
where \(F=(F^a)\) is the induced holomorphic change of coordinates on
\(D\), \(u\) is holomorphic and nonvanishing, and \(W=(W^a)\) is
holomorphic. 

The tangential term transforms by \(A' = A \circ F\), so \( i\partial\bar\partial A'=F^*(i\partial\bar\partial A)\), and therefore \(H\) defines intrinsically a K\"ahler metric \(h_D\) on \(TD\).
Substituting the coordinate change into
\eqref{eq:admissible-potential} and comparing the coefficients of
\(\bar w\) and \(|w|^{2\beta}\) gives
\[
 Q'=\overline{u}\,Q+\overline{W^b}\,\partial_{\bar b}A,
 \qquad
 s'=|u|^2s.
\]
Differentiating the first identity, and writing
\(F^a_\alpha=\partial F^a/\partial x^\alpha\), gives
\[
 q'_\alpha
 =F^a_\alpha
 \bigl(\overline{u}\,q_a
       +H_{a\bar b}\overline{W^b}\bigr).
\]
Since \([\partial_w]=u[\partial_z]\) in \(N_D\), the transformation law
for \(s\) defines a Hermitian metric \(h_N\) on \(N_D\) by
\[
 |[\partial_z]|_{h_N}^2=s.
\]

Finally, along \(D\),
\[
 \partial_w=u\,\partial_z+W^a\partial_{y^a},
 \qquad
 \partial_{x^\alpha}=F^a_\alpha\partial_{y^a}.
\]
Using the transformation laws for \(H\) and \(q\) in
\eqref{eq:Lg-local} gives \(\ell'=u\ell\). Thus the line spanned by
\(\ell\) defines intrinsically a smooth complex line subbundle
\(L_g\subset TX|_D\). Since \([\ell]=[\partial_z]\neq0\) in \(N_D\),
it is transverse to \(TD\).
\end{proof}

\begin{proposition}\label{prop:boundary-hermitian-metric}
The boundary data of Proposition~\ref{prop:boundary-data} determine a
unique Hermitian metric \(h\) on \(TX|_D\) such that
\begin{equation}\label{eq:boundary-metric-geometric}
 h|_{TD}=h_D,
 \qquad
 TD\perp_h L_g,
 \qquad
 h|_{L_g}=\pi^*h_N,
\end{equation}
where \(\pi:L_g\to N_D\) is the quotient isomorphism. In the coordinate
frame \((\partial_z,\partial_{y^a})\),
\begin{equation}\label{eq:h-matrix}
 h=
 \begin{pmatrix}
 s+q^*H^{-1}q&q^*\\
 q&H
 \end{pmatrix},
\end{equation}
where \(q=(q_a)\) is regarded as a column vector.
\end{proposition}

\begin{proof}
Since \(L_g\) is transverse to \(TD\), the quotient map restricts to an
isomorphism \(\pi:L_g\to N_D\). Thus \eqref{eq:boundary-metric-geometric}
defines \(h\) uniquely by declaring \(TD\) and \(L_g\) orthogonal and
using \(h_D\) and \(\pi^*h_N\) on the two summands.

To obtain the matrix formula, set
\[
 \nu^a=H^{a\bar b}\overline{q_b}.
\]
Then
\[
 \ell=\partial_z-\nu^a\partial_{y^a},
 \qquad
 \partial_z=\ell+\nu^a\partial_{y^a}.
\]
By construction, \(\ell\perp_h TD\) and
\[
 |\ell|_h^2=|[\ell]|_{h_N}^2
 =|[\partial_z]|_{h_N}^2=s.
\]
It follows that
\[
 h_{a\bar b}=H_{a\bar b},
 \qquad
 h_{a\bar1}
 =H_{a\bar b}\overline{\nu^b}
 =q_a,
\]
and
\[
 h_{1\bar1}
 =s+H_{a\bar b}\nu^a\overline{\nu^b}
 =s+q^*H^{-1}q.
\]
This is precisely \eqref{eq:h-matrix}.
\end{proof}

\begin{remark}
The line \(L_g\) is the limit, as \(z\to0\), of the \(g\)-orthogonal complements of the hypersurfaces \(\{z=\mathrm{constant}\}\); see \cite[\S 6.2]{deBorbonThesis}.
\end{remark}

\begin{definition}[Boundary Hermitian metric]
\label{def:boundary-metric}
The metric \(h\) of Proposition~\ref{prop:boundary-hermitian-metric} is
called the \emph{boundary Hermitian metric} associated with \(g\). Its
\emph{second fundamental form} is
\begin{equation}\label{eq:II-def}
 \II_h(U,V)=\pi_{L_g}\bigl(\nabla^h_UV\bigr),
 \qquad U,V\in T^{1,0}D,
\end{equation}
where \(\nabla^h\) is the Chern connection of \((TX|_D,h)\) and
\(\pi_{L_g}\) is the \(h\)-orthogonal projection onto \(L_g\). 
\end{definition}

\subsection{Adapted normal coordinates}\label{sec:bochner}

Now we use the Bochner construction \cite{Bochner1947} which removes low-order Taylor terms of a K\"ahler potential by holomorphic coordinate changes and pluriharmonic changes of potential. Along a divisor the allowed coordinate group is smaller: an adapted change must preserve \(D\), so its normal component has the form
\begin{equation}\label{eq:adapted-change}
 z'=z\,u(z,y),
 \qquad
 y'=F(z,y),
 \qquad
 u|_D\neq0.
\end{equation}
In particular, the new normal coordinate cannot contain a pure tangential quadratic term. The first tensor left by this missing freedom is the second fundamental form.

\begin{lemma}[Coordinate invariance of admissibility]\label{lem:stability-admissibility}
Let \(g\) be admissible in adapted coordinates \((z,y)\), with \(\eps\) satisfying \eqref{eq:epsilon-range}. Then, after any divisor-preserving holomorphic coordinate change
\begin{equation}\label{eq:general-adapted-change}
 z=w\,u(w,x),\qquad y=F(w,x),\qquad u(0,x)\neq0,
\end{equation}
the potential can again be written in the form \eqref{eq:admissible-potential}, and the remainder satisfies \eqref{eq:admissible-remainder} with the same \(\eps\).
\end{lemma}

\begin{proof}
Set \(\lambda=2\beta+\beta\eps\), so \(1<\lambda<2\). Substituting \eqref{eq:general-adapted-change} into \(A+\bar zQ+z\overline Q\) produces new constant and normal-linear coefficients \(A'\) and \(Q'\), together with a smooth remainder \(O_{\mathrm{con}}(|w|^2)\), hence \(O_{\mathrm{con}}(|w|^\lambda)\).

For the cone term, set
\[
 S=s(F,\overline F)|u|^2,
 \qquad
 s'(x,\bar x)=S(0,0,x,\bar x).
\]
Then
\[
 \frac{s^\beta}{\beta^2}|z|^{2\beta}
 =\frac{S^\beta}{\beta^2}|w|^{2\beta},
\]
and \(S^\beta-(s')^\beta=O_{\mathrm{con}}(|w|)\). Thus the resulting error is \(O_{\mathrm{con}}(|w|^{2\beta+1})\), hence \(O_{\mathrm{con}}(|w|^\lambda)\).

It remains to transform \(\Psi\). The conormal vector fields in the \((w,x)\)-coordinates are smooth linear combinations of those in the \((z,y)\)-coordinates. For example,
\[
 w\partial_w
 =\frac{u+w\partial_wu}{u}\,z\partial_z
  +w(\partial_wF^a)\partial_{y^a},
 \qquad
 \partial_{x^c}
 =\frac{\partial_{x^c}u}{u}\,z\partial_z
  +(\partial_{x^c}F^a)\partial_{y^a},
\]
with analogous conjugate formulas; the inverse transition has the same property. Since \(|z|\) and \(|w|\) are uniformly comparable near \(D\), \eqref{eq:admissible-remainder} therefore gives
\[
 \Psi\circ(z,y)=O_{\mathrm{con}}(|w|^\lambda).
\]
Combining the three remainder terms proves the claim. Positivity of the new tangential Hessian and of \(s'\) follows from the corresponding properties of \(H\) and \(s\).
\end{proof}

The construction below follows Bochner's method for K\"ahler normal coordinates \cite{Bochner1947}, with the additional requirement that the divisor \(D\) be preserved.

\begin{proposition}[Adapted normal coordinates]\label{prop:relative-Bochner}
Let \(g\) be admissible and let \(p\in D\). There are adapted holomorphic
coordinates \((z,y)\) centred at \(p\) and a pluriharmonic choice of the
local potential \eqref{eq:admissible-potential} such that
\begin{equation}\label{eq:first-normalization}
 H(p)=I,
 \qquad
 q(p)=0,
 \qquad
 s(p)=1,
 \qquad
 \partial_cH_{a\bar b}(p)=0,
 \qquad
 \partial_cs(p)=0,
\end{equation}
and
\begin{equation}\label{eq:A-Bochner}
 A(y,\bar y)=|y|^2+O(|y|^4).
\end{equation}
Moreover, the quadratic part of \(Q\) can be normalized to
\begin{equation}\label{eq:Q-Bochner}
 Q(y,\bar y)=\frac12B_{ab}y^ay^b+O(|y|^3).
\end{equation}
Since \(q(p)=0\), we have \(L_g(p)=\C\partial_z\). Define the scalar
components \((\II_h)_{ab}(p)\) by
\[
 \II_h(\partial_{y^a},\partial_{y^b})\big|_p
 =(\II_h)_{ab}(p)\,\partial_z.
\]
Then
\begin{equation}\label{eq:B-II}
 B_{ab}=(\II_h)_{ab}(p).
\end{equation}
In particular, \(\II_h\) is symmetric in its two tangential arguments
and, using the natural identification \(L_g\simeq N_D\), may be regarded
as a section of \(\Sym^2T^{*1,0}D\otimes N_D\).
\end{proposition}

\begin{proof}
Choose K\"ahler normal holomorphic coordinates \(y\) for
\((D,h_D)\) at \(p\). After adding a tangential pluriharmonic function
to the potential, we may assume
\[
 H(p)=I,
 \qquad
 \partial_cH_{a\bar b}(p)=0,
 \qquad
 A(y,\bar y)=|y|^2+O(|y|^4).
\]
Choose a local holomorphic frame of \(N_D\) which is Chern-normal for
\(h_N\) at \(p\), and choose the defining function \(z\) so that it
induces this frame. We may then arrange
\[
 s(p)=1,
 \qquad
 \partial_cs(p)=0.
\]
Finally, replace the tangential coordinates by
\[
 y^a\longmapsto y^a+z\lambda^a
\]
for suitable constants \(\lambda^a\), so that the coordinate normal line
agrees with \(L_g\) at \(p\). By \eqref{eq:h-matrix}, this is equivalent
to \(q(p)=0\). Thus \eqref{eq:first-normalization} and
\eqref{eq:A-Bochner} hold.

It remains to normalize the terms of \(Q\) of total degree at most two.
Consider a holomorphic shear
\begin{equation}\label{eq:shear-proof}
 z'=z,
 \qquad
 y'^a=y^a+zV^a(y),
\end{equation}
where \(V\) is holomorphic. Substituting the inverse change of coordinates
into \(A(y,\bar y)\), the coefficient of \(\bar z\) changes, through
total tangential degree two, by
\[
 Q-A_{\bar a}\,\overline{V^a(y)}+O(|y|^3).
\]
Independently, we may change the K\"ahler potential by an arbitrary
pluriharmonic function of the form
\[
 zf_1(y)+\overline{zf_1(y)},
\]
where \(f_1\) is holomorphic. Combining the shear with this change of
potential, the new coefficient of \(\bar z\) is therefore
\begin{equation}\label{eq:Q-transform-proof}
 Q'
 =
 Q-A_{\bar a}\,\overline{V^a(y)}
 +\overline{f_1(y)}
 +O(|y|^3).
\end{equation}

By \eqref{eq:A-Bochner},
\[
 A_{\bar a}=y^a+O(|y|^3).
\]
Since \(q(p)=0\), the holomorphic linear part of \(Q\) vanishes, and we
may take \(V(0)=0\). The linear part of \(V\) removes the quadratic terms
of bidegree \((1,1)\) in \(Q\), while \(f_1\) removes the constant,
antiholomorphic linear, and purely antiholomorphic quadratic terms.
Thus the only term of degree at most two which remains is a purely
holomorphic quadratic polynomial, giving \eqref{eq:Q-Bochner}.

Taking two holomorphic tangential derivatives of
\eqref{eq:Q-transform-proof} at \(p\) shows that this holomorphic
quadratic coefficient is unaffected by either the shear or the
pluriharmonic change of potential. Hence
\begin{equation}\label{eq:B-Q-proof}
 B_{ab}=\partial_a\partial_bQ(p).
\end{equation}
By Lemma~\ref{lem:stability-admissibility}, admissibility and the
conormal estimate \eqref{eq:admissible-remainder} remain valid in the
coordinates constructed above.

It remains to identify \(B_{ab}\) with the second fundamental form.
Let
\[
 e_1=\partial_z,
 \qquad
 e_a=\partial_{y^a}
\]
be the induced holomorphic frame of \(TX|_D\), and write
\[
 \nabla^h_{\partial_{y^a}}e_b
 =\Gamma^k_{ab}e_k.
\]
Metric compatibility of the Chern connection gives
\[
 \partial_a h_{b\bar j}
 =\Gamma^k_{ab}h_{k\bar j},
\]
and therefore
\[
 \Gamma^k_{ab}
 =h^{k\bar j}\partial_a h_{b\bar j}.
\]
At \(p\), \eqref{eq:first-normalization} and \eqref{eq:h-matrix} give
\(h_{i\bar j}(p)=\delta_{ij}\), while \(q(p)=0\) gives
\(L_g(p)=\C\partial_z\). Hence the \(L_g\)-component of
\(\nabla^h_{\partial_{y^a}}\partial_{y^b}\) at \(p\) is
\[
 \Gamma^1_{ab}(p)\partial_z
 =\partial_a h_{b\bar1}(p)\partial_z.
\]
By the definition of the scalar components \((\II_h)_{ab}(p)\), this
means
\[
 (\II_h)_{ab}(p)=\partial_a h_{b\bar1}(p).
\]
Since \eqref{eq:h-matrix} gives
\(h_{b\bar1}=q_b=\partial_bQ\) along \(D\), we conclude that
\[
 (\II_h)_{ab}(p)
 =\partial_a\partial_bQ(p)
 =B_{ab},
\]
which proves \eqref{eq:B-II}. Since \(B_{ab}=B_{ba}\), the second
fundamental form is symmetric.

Finally, the identity \(B_{ab}=(\II_h)_{ab}(p)\) shows that \(B_{ab}\)
is independent of the remaining freedom in the choice of adapted normal
coordinates.
\end{proof}

\begin{remark}[The remaining coordinate freedom]\label{rem:freedom}
	Adapted normal coordinates are not unique. Their linear freedom is the natural
	\(U(1)\times U(n-1)\) action. Through degree two, the normalization fixes every
	term involving a tangential direction, while purely normal quadratic terms
	remain free. No normalization is imposed on terms of degree three or higher,
	apart from the requirement that the divisor be preserved.
\end{remark}

\section{Curvature and holomorphic splitting}\label{sec:curvature}

We use the K\"ahler curvature convention
\begin{equation}\label{eq:curvature-convention}
 R_{i\bar j k\bar\ell}
 =-\partial_i\partial_{\bar j}g_{k\bar\ell}
 +g^{p\bar q}(\partial_i g_{k\bar q})(\partial_{\bar j}g_{p\bar\ell}).
\end{equation}
\begin{theorem}[Mixed curvature growth]\label{thm:main}
Let \(D\subset X\) be a smooth divisor, let \(1/2<\beta<1\), and let \(g\) be an admissible K\"ahler metric of cone angle \(2\pi\beta\) along \(D\). Fix \(p\in D\) and take adapted normal coordinates
\[
 (z,y^2,\ldots,y^n),\qquad D=\{z=0\},\qquad p=(0,0),
\]
as in Proposition~\ref{prop:relative-Bochner}. Then, along \(y=0\) and as \(z\to0\),
\begin{equation}\label{eq:main-coordinate}
 R_{a\bar1b\bar1}
 =\frac{\beta-1}{\bar z}(\II_h)_{ab}(p)+o(|z|^{-1}).
\end{equation}
Here the index \(1\) denotes the \(z\)-direction and \(a,b\) are tangential indices.
\end{theorem}

Fix \(p\in D\) and the coordinates of Proposition~\ref{prop:relative-Bochner}. By Lemma~\ref{lem:stability-admissibility}, the conormal remainder estimate \eqref{eq:admissible-remainder} remains valid in these normalized coordinates. All estimates below are taken along \(y=0\), as \(z\to0\).

\begin{lemma}[Metric and first-derivative estimates]\label{lem:metric-estimates}
In the adapted normal coordinates of Proposition~\ref{prop:relative-Bochner}, one has
\begin{align}
 g_{1\bar1}
 &=\rho^{2\beta-2}\bigl(1+O(\rho^{\beta\eps})\bigr),\label{eq:g11}\\
 g_{a\bar b}
 &=\delta_{ab}+o(1),\label{eq:gab}\\
 g_{1\bar a}
 &=O\bigl(\rho^{2\beta+\beta\eps-1}\bigr),\label{eq:g1a}\\
 g^{1\bar a}
 &=O\bigl(\rho^{1+\beta\eps}\bigr),\label{eq:ginv1a}
\end{align}
and
\begin{align}
 \partial_a g_{b\bar1}
 &=(\II_h)_{ab}(p)
   +O(\rho^{2\beta-1})
   +O(\rho^{2\beta+\beta\eps-1}),\label{eq:dagb1}\\
 \partial_a g_{b\bar d}
 &=O(\rho),\label{eq:dagbd}\\
 \partial_{\bar z}g_{c\bar1}
 &=O\bigl(\rho^{2\beta+\beta\eps-2}\bigr),\label{eq:dbargc1}\\
 \partial_{\bar z}g_{1\bar1}
 &=O(\rho^{2\beta-3}).\label{eq:dbarg11}
\end{align}
Finally,
\begin{equation}\label{eq:singular-log-derivative}
 g^{1\bar1}\partial_{\bar z}g_{1\bar1}
 =\frac{\beta-1}{\bar z}+o(\rho^{-1}).
\end{equation}
\end{lemma}

\begin{proof}
The normal--normal component is obtained by applying \(\partial_z\partial_{\bar z}\) to \eqref{eq:admissible-potential}. Since
\[
 \partial_z\partial_{\bar z}\frac{|z|^{2\beta}}{\beta^2}
 =\rho^{2\beta-2},
\]
and \(s(p)=1\), the singular term contributes \(\rho^{2\beta-2}\). Two ordinary normal derivatives of the remainder contribute \(O(\rho^{2\beta+\beta\eps-2})\), proving \eqref{eq:g11}.

The tangential component is
\[
 g_{a\bar b}
 =H_{a\bar b}+O(\rho)+O(\rho^{2\beta})+O(\rho^{2\beta+\beta\eps}),
\]
so \eqref{eq:gab} follows from \(H(p)=I\).

For the mixed component, differentiating first in \(\bar z\) gives
\[
 g_{b\bar1}
 =\partial_bQ
 +s^{\beta-1}(\partial_b s)\,z\rho^{2\beta-2}
 +O(\rho^{2\beta+\beta\eps-1}).
\]
Along \(y=0\), the first term is \(q_b(p)=0\) and the second vanishes because \(\partial_b s(p)=0\). This proves \eqref{eq:g1a} and its conjugate.

Block inversion of \eqref{eq:g11}--\eqref{eq:g1a} gives
\[
 g^{1\bar1}=\rho^{2-2\beta}\bigl(1+O(\rho^{\beta\eps})\bigr),
 \qquad
 g^{1\bar a}=O\bigl(\rho^{1+\beta\eps}\bigr),
\]
which proves \eqref{eq:ginv1a}.

Differentiating the displayed formula for \(g_{b\bar1}\) tangentially gives
\[
 \partial_a g_{b\bar1}
 =\partial_a\partial_bQ(p)
  +O(\rho^{2\beta-1})
  +O(\rho^{2\beta+\beta\eps-1}).
\]
By Proposition~\ref{prop:relative-Bochner}, \(\partial_a\partial_bQ(p)=(\II_h)_{ab}(p)\), giving \eqref{eq:dagb1}.

For \eqref{eq:dagbd}, the derivative of \(H_{b\bar d}\) vanishes at \(p\) by \eqref{eq:first-normalization}. The \(zQ+\bar z\bar Q\) terms contribute \(O(\rho)\), the cone term contributes \(O(\rho^{2\beta})\), and the remainder contributes \(O(\rho^{2\beta+\beta\eps})\). Since \(2\beta>1\), all are \(O(\rho)\).

In \(g_{c\bar1}\), the leading term \(\partial_cQ\) is independent of \(z\). The cone contribution is proportional to \(\partial_cs(p)\) and hence vanishes on \(y=0\). Applying one further \(\bar z\) derivative to the remainder gives \eqref{eq:dbargc1}. Similarly, differentiating \eqref{eq:g11} once in \(\bar z\) gives \eqref{eq:dbarg11}.

For the final formula, the leading term of \(\partial_{\bar z}g_{1\bar1}\) is
\[
 (\beta-1)z\rho^{2\beta-4}.
\]
Multiplying by \(g^{1\bar1}=\rho^{2-2\beta}(1+o(1))\) gives
\[
 (\beta-1)z\rho^{-2}+o(\rho^{-1})
 =\frac{\beta-1}{\bar z}+o(\rho^{-1}),
\]
which is \eqref{eq:singular-log-derivative}.
\end{proof}

\begin{lemma}[Order separation in the curvature formula]\label{lem:order-separation}
In the coordinates of Proposition~\ref{prop:relative-Bochner},
\begin{equation}\label{eq:second-derivative-small}
 -\partial_a\partial_{\bar z}g_{b\bar1}
 =O(\rho^{2\beta-2})
  +O(\rho^{2\beta+\beta\eps-2})
 =o(\rho^{-1}),
\end{equation}
while
\begin{equation}\label{eq:quadratic-main}
 g^{1\bar1}(\partial_a g_{b\bar1})(\partial_{\bar z}g_{1\bar1})
 =\frac{\beta-1}{\bar z}(\II_h)_{ab}(p)+o(\rho^{-1}).
\end{equation}
The sum of the remaining quadratic terms satisfies
\begin{equation}\label{eq:quadratic-remaining}
 \sum_{(p,q)\neq(1,1)}
 g^{p\bar q}(\partial_a g_{b\bar q})(\partial_{\bar z}g_{p\bar1})
 =o(\rho^{-1}).
\end{equation}
Consequently,
\begin{equation}\label{eq:coordinate-curvature-asymptotic}
 R_{a\bar1b\bar1}
 =\frac{\beta-1}{\bar z}(\II_h)_{ab}(p)+o(\rho^{-1}).
\end{equation}
\end{lemma}

\begin{proof}
The term \(\partial_bQ\) in \(g_{b\bar1}\) is independent of \(z\) and is therefore annihilated by \(\partial_{\bar z}\). The cone contribution to \(\partial_a\partial_{\bar z}g_{b\bar1}\) has order \(O(\rho^{2\beta-2})\), while the ordinary normal derivative estimates following from \eqref{eq:admissible-remainder} give \(O(\rho^{2\beta+\beta\eps-2})\) for the remainder. Since \(\beta>1/2\), both exponents are strictly greater than \(-1\), proving \eqref{eq:second-derivative-small}.

Combining \eqref{eq:dagb1} with \eqref{eq:singular-log-derivative} gives
\[
 g^{1\bar1}(\partial_a g_{b\bar1})(\partial_{\bar z}g_{1\bar1})
 =\left((\II_h)_{ab}(p)+o(1)\right)
  \left(\frac{\beta-1}{\bar z}+o(\rho^{-1})\right),
\]
which is \eqref{eq:quadratic-main}.

It remains to estimate the other choices of \((p,q)\) in the quadratic term of \eqref{eq:curvature-convention}. If \(p,q\) are both tangential, \eqref{eq:dagbd} and \eqref{eq:dbargc1} give
\[
 O(1)\,O(\rho)\,O(\rho^{2\beta+\beta\eps-2})
 =O(\rho^{2\beta+\beta\eps-1}).
\]
If \(p=1\) and \(q\) is tangential, \eqref{eq:ginv1a}, \eqref{eq:dagbd}, and \eqref{eq:dbarg11} give
\[
 O(\rho^{1+\beta\eps})\,O(\rho)\,O(\rho^{2\beta-3})
 =O(\rho^{2\beta+\beta\eps-1}).
\]
If \(p\) is tangential and \(q=1\), the conjugate of \eqref{eq:ginv1a}, together with \eqref{eq:dagb1} and \eqref{eq:dbargc1}, gives
\[
 O(\rho^{1+\beta\eps})\,O(1)\,O(\rho^{2\beta+\beta\eps-2})
 =O(\rho^{2\beta+2\beta\eps-1}).
\]
Every exponent is greater than \(-1\), proving \eqref{eq:quadratic-remaining}. Substitution into \eqref{eq:curvature-convention} gives \eqref{eq:coordinate-curvature-asymptotic}.
\end{proof}

\begin{proof}[Proof of Theorem~\ref{thm:main}]
Lemma~\ref{lem:order-separation} gives exactly \eqref{eq:main-coordinate}.
\end{proof}

To read \eqref{eq:main-coordinate} as a curvature bound, write \(z=\rho e^{i\theta}\) and set
\[
 V_1=\rho^{1-\beta}\partial_z,
 \qquad
 V_a=\partial_{y^a}.
\]
By \eqref{eq:g11}--\eqref{eq:g1a} the Gram matrix of \((V_1,V_2,\ldots,V_n)\) tends to the identity as \(z\to0\) along \(y=0\), so this frame is asymptotically unitary, and \(R(V_a,\overline{V_1},V_b,\overline{V_1})=\rho^{2-2\beta}R_{a\bar1b\bar1}\). Since \(\bar z=\rho e^{-i\theta}\), Theorem~\ref{thm:main} therefore gives
\begin{equation}\label{eq:main-bounded-frame}
 R(V_a,\overline{V_1},V_b,\overline{V_1})
 =-(1-\beta)e^{i\theta}\rho^{1-2\beta}(\II_h)_{ab}(p)
 +o\bigl(\rho^{1-2\beta}\bigr).
\end{equation}

\begin{corollary}[Curvature growth]\label{cor:growth}
Under the assumptions of Theorem~\ref{thm:main}, suppose that \((\II_h)_p\neq0\). Then the curvature diverges as \(z\to0\) along \(y=0\):
\[
 \lim_{z\to0}|\Rm(g)|_g=+\infty.
\]
More precisely,
\begin{equation}\label{eq:growth-liminf}
 \liminf_{z\to0}|z|^{2\beta-1}|\Rm(g)|_g>0.
\end{equation}
\end{corollary}

\begin{proof}
Choose \(a,b\) with \((\II_h)_{ab}(p)\neq0\). In the asymptotically unitary frame above, \eqref{eq:main-bounded-frame} gives
\[
 \rho^{2\beta-1}\bigl|R(V_a,\overline{V_1},V_b,\overline{V_1})\bigr|
 \longrightarrow
 (1-\beta)\bigl|(\II_h)_{ab}(p)\bigr|>0,
\]
so \eqref{eq:growth-liminf} holds. Since \(2\beta-1>0\), the curvature is unbounded.
\end{proof}

\begin{corollary}[Holomorphic splitting]\label{cor:splitting}
Under the assumptions of Theorem~\ref{thm:main}, if \(|\Rm(g)|_g\) is locally bounded near \(D\), then the sequence \eqref{eq:normal-sequence-intro} splits holomorphically.
\end{corollary}

\begin{proof}
By Corollary~\ref{cor:growth}, bounded curvature implies \(\II_h\equiv0\). By the standard theory of Hermitian holomorphic subbundles \cite{Kobayashi1987}, the \(h\)-orthogonal complement \(TD^{\perp_h}=L_g\) is then a holomorphic subbundle of \(TX|_D\), so the normal bundle sequence splits holomorphically.
\end{proof}

\subsection{Remarks}

\begin{remark}
	 The statement of Corollary~\ref{cor:splitting} applies in particular to K\"ahler--Einstein metrics with cone singularities; cf.~Remark~\ref{rem:JMR}. 
\end{remark}

\begin{remark}\label{rem:SongWang}
	For polyhomogeneous conical K\"ahler metrics with bounded Ricci curvature, Song--Wang \cite[Theorem~4.6]{SongWang2016} prove the matching upper bound \(|\Rm(g)|_g=O(\rho^{1-2\beta})\); thus, in that setting, the exponent in Corollary~\ref{cor:growth} is sharp whenever \(\II_h\neq0\).	
\end{remark}

\begin{remark}
	Consider the Calabi--Yau metric on \(\C^2\) with cone angle \(2\pi\beta\) along the smooth conic \(C=\{zw=1\}\) obtained by Donaldson \cite[\S5]{Donaldson2012} using the Gibbons--Hawking ansatz. For \(\beta>1/2\) its curvature is unbounded near \(C\); see \cite[\S3.3]{deBorbonGH}. Moreover, the discussion in \cite[\S4]{deBorbonGH} shows that the sectional curvature remains bounded as \(C\) is approached in the tangential and in the normal directions, while it is unbounded in the intermediate directions. This gives a concrete model for the mixed nature of the curvature growth in \eqref{eq:main-bounded-frame}.
\end{remark}

\begin{remark}\label{rem:codazzi}
	Equation~\eqref{eq:main-coordinate} shows that the mixed curvature is governed
	to leading order by the second fundamental form \(\II_h\) of \(D\). This is
	reminiscent of the Codazzi equation in the smooth setting \cite[Chapter~VII, \S4]{KobayashiNomizuII}.
\end{remark}

\section*{AI declaration}
The idea of relating bounded curvature to the vanishing of the second fundamental form
\(\II_h\) of \(TD\subset TX|_D\), equipped with the boundary Hermitian metric \(h\), is due to
the author. While preparing a grant application for EPSRC, I discussed this problem with
ChatGPT. In the course of that discussion, ChatGPT derived the local curvature asymptotic \eqref{eq:main-coordinate}, which provided the key step in the proof. I subsequently
checked and developed the argument presented here.

\bibliographystyle{plain}
\bibliography{ref}

\end{document}